\documentclass[11pt]{article}
\usepackage{authblk}
\usepackage{currfile}
\usepackage{floatrow}
\usepackage{amsmath,amsthm,verbatim,amssymb,amsfonts,amscd}
\usepackage{graphicx,tikz,caption,subfig}
\usepackage{mathrsfs}
\usepackage{amsmath}

\usepackage{tikz}
\usetikzlibrary{decorations.markings}
\tikzset{>=stealth}
\tikzstyle{vertex}=[circle, draw, inner sep=1pt, minimum size=1pt]
\tikzstyle{directed}=[postaction={decorate,
	decoration={markings,mark=at position 0.3 with {\arrow{stealth}}}
}]

\newcommand{\vertex}{\node[vertex]}

\usepackage{bm}
\usepackage{float}
\usepackage{setspace}
\usetikzlibrary{arrows,shapes,positioning}
\usetikzlibrary{decorations.markings}
\tikzstyle arrowstyle=[scale=0.8]
\tikzstyle directed=[postaction={decorate,decoration={markings, mark=at position 0.55 with {\arrow[arrowstyle]{stealth}}}}]
\tikzstyle redirected=[postaction={decorate,decoration={markings, mark=at position 0.75 with {\arrow[arrowstyle]{stealth reversed}}}}]

\usepackage{xcolor}

\usepackage{enumitem}
\setenumerate[1]{itemsep=0pt,partopsep=0pt,parsep=\parskip,topsep=3pt}
\setitemize[1]{itemsep=0pt,partopsep=0pt,parsep=\parskip,topsep=3pt}
\setdescription{itemsep=0pt,partopsep=0pt,parsep=\parskip,topsep=3pt}

\newtheorem{claim}{Claim}
\newtheorem{theorem}{Theorem}[section]

\newtheorem{definition}[theorem]{Definition}
\newtheorem{conjecture}[theorem]{Conjecture}

\newtheorem{lemma}[theorem]{Lemma}

\makeatletter \@addtoreset{equation}{section} \makeatother

\newcommand{\AMSE}{{\it Acta Math. Sin. (Engl. Ser.)}}

\newcommand{\JCTB}{{\it J. Combin. Theory Ser. B}}

\newcommand{\JGT}{{\it J. Graph Theory}}

\newcommand{\DM}{{\it Discrete Math.}}

\newcommand{\SIAMDM}{{\it SIAM J. Discrete Math.}}

\newcommand{\CPC}{{\it Combin. Probab. Comput.}}
\newcommand{\GC}{{\it Graphs Comb.}}

\newcommand{\JCO}{{\it J. Comb. Optim.}}

\newcommand{\IC}{{\it Inform. Comput.}}

\begin{document}
\title{Signed list edge coloring in graphs of bounded treewidth}

\author{Li Zhang\thanks{Yan'an University, Yan'an, Shaanxi, 716000, China. Email:  lizhang@yau.edu.cn.},
~~You Lu\thanks{School of Mathematics and Statistics, Northwestern Polytechnical University,  Xi'an, Shaanxi 710129, China. Email: luyou@nwpu.edu.cn. Partially supported by National Natural Science Foundation of China (No. 12271438).},
~~Zhengke Miao\thanks{School of Mathematics and Statistics \& Key Laboratory of Analytical Mathematics and Applications (Ministry of Education), Fujian Normal University, Fuzhou, Fujian 350007, China. Email:  zkmiao@jsnu.edu.cn. Partially supported by National Natural Science Foundation of China (No. 12431013).},
~~Yintao Wang\thanks{School of Marine Science and Technology, Northwestern Polytechnical University,  Xi'an, Shaanxi 710072, China. Email: tyaowang@gmail.com. Partially supported by National Natural Science Foundation of China (No. U2141238).}
}

\date{}
\maketitle

\begin{abstract}
Vizing conjectured that the list edge chromatic number of any graph with maximum degree $\Delta$ is at most $\Delta + 1$. This conjecture has been confirmed for several important classes of graphs, in particular, Lang proved that it holds for all graphs of treewidth $3$. In this paper, we introduce the list edge coloring of signed graphs, a framework that generalizes both classical list edge coloring and the signed edge coloring introduced by Behr. We extend Lang's result by proving the signed analogue of Vizing's conjecture for all signed graphs of treewidth $3$, as well as for signed graphs of treewidth $4$ with maximum degree $\Delta \ge 10$.

\end{abstract}

\section{Introduction}

Graphs or signed graphs considered in this paper are finite and simple. For terminology and notation not defined here, we follow \cite{BondyMurty2008, CCJST2019, Diestel2010}.

Let $G$ be a graph and $\mathbb{Z}$ be the set of integers. A {\em list edge assignment} on $G$ is a mapping $L$ that assigns to each edge $e\in E(G)$ a {\em list} $L(e)\subseteq \mathbb{Z}$ of colors.  An {\em $L$-edge-coloring} of $G$ is a mapping $\phi: E(G)\to \mathbb{Z}$ such that every edge $e\in E(G)$ receives a color $\phi(e)\in L(e)$ and $\phi(e)\neq \phi(e')$ for any two adjacent edges $e, e'\in E(G)$.
The {\em list edge chromatic number}, denoted by $\chi'_l(G)$, is the smallest  positive integer $k$ such that $G$ admits an $L$-edge-coloring for any list edge assignment $L$ satisfying $|L(e)|\ge k$ for all $e\in E(G)$.

The classical $k$-edge-coloring is a special case of $L$-edge-coloring where all lists are equal to $\{1, 2, \dots, k\}$. Thus $\chi_l'(G)\ge \chi'(G)\ge \Delta$ for any graph $G$, where $\chi'(G)$ denotes the classical edge chromatic number and $\Delta$ is the maximum degree of $G$.
A fundamental result in this area, known as Vizing's theorem and independently proved by Vizing \cite{Vizing1964} and Gupta \cite{Gupta}, states that $\chi'(G)\leq \Delta+1$ for every graph $G$. In 1976, Vizing proposed the following conjecture.

\begin{conjecture}[Vizing \cite{Vizing1976}]\label{con: list edge}
$\chi_l'(G)\leq \Delta+1$ for every graph $G$.
\end{conjecture}

Though still open, Conjecture \ref{con: list edge} has attracted considerable attention. It was initially verified by Vizing \cite{Vizing1976} for graphs with maximum degree at most $3$, a result independently established by Erd\H{o}s, Rubin and Taylor \cite{ERT1980}. Juvan, Mohar and \v{S}krekovski \cite{JMS1998} extended  this result to graphs with maximum degree $4$. Lang  \cite{Lang2016} further confirmed the conjecture for all graphs of treewidth $3$.
For additional related results, we refer the readers to
\cite{MB2015, Bruhn2015, Galvin1995, H1997,  H2021, Meeks2016,  W2018, Z2025}.

Recently, Behr \cite{B2020} generalized the concept of edge coloring in a natural way to signed graphs and established a signed analogue of Vizing's theorem. Building upon Behr's definition, we introduce in this paper the notion of signed list edge coloring.

A {\em signed graph} $(G,\sigma)$ consists of an {\em underlying graph} $G$ and a {\em signature} $\sigma: E(G)\to \{1,-1\}$.
An edge $e\in E(G)$ is {\em positive} if $\sigma(e)=1$ and {\em negative} otherwise.  For a subgraph $(G', \sigma|_{E(G')}) $ (or simply $(G', \sigma)$) of $(G, \sigma)$, we use $E_N(G',\sigma)$ to denote the set of negative edges of $(G', \sigma)$. A cycle is \emph{balanced} if it contains an even number of negative edges, and otherwise \emph{unbalanced}.
Following Bouchet \cite{B1983}, every edge of $G$ is viewed as two \emph{half-edges} $h$ and $\hat{h}$, one incident with each end. For clarity, if $v$ is an end of $e$, we denote by $h^{v}_e$ the
half-edge of $e$ incident with $v$. Let $H(G)$ be the set of all half-edges of $G$,  and $H_G(v)$ be the set of half-edges incident with a vertex $v$. For a half-edge $h \in H(G)$, we write $e_h$ for the edge containing $h$.

Let $\mathbb{Z}$ be the set of integers and let $2^{\mathbb Z}$ denote the family of all subsets of $\mathbb Z$. For any $S\in 2^{\mathbb{Z}}$, define $-S = \{-s : s \in S\}$.
A \emph{list edge assignment} on a signed graph $(G, \sigma)$ is a mapping $L : H(G)\rightarrow 2^{\mathbb{Z}}$ such that
$L(h) = -\sigma(e_h)L(\hat{h})$ for every $h \in H(G)$.

\begin{definition}
 Let $(G,\sigma)$ be a signed graph with a list edge assignment $L$, and $f: E(G)\to \mathbb{Z}^+$ be a function.

 \begin{itemize}
 \item[\rm (1)] An {\em $L$-edge-coloring} of $(G,\sigma)$ is a mapping $\phi : H(G)\rightarrow \mathbb{Z}$ such that $\phi(h) =-\sigma(e_h)\phi(\hat{h}) \in L(h)$ and
 $\phi(h) \neq \phi(h')$ for any adjacent half-edges $h, h'\in H(G)$.

\item[\rm (2)] $(G,\sigma)$ is called {\em $f$-list-edge-colorable} if it admits an $L$-edge-coloring for every list edge assignment $L$
satisfying $|L(h)|\geq f(e_h)$ for all $h \in H(G)$. In particular, $(G, \sigma)$ is {\em $k$-list-edge-colorable} if it is $f$-list-edge-colorable with $f(e) = k$ for every $e \in E(G)$.

\item[\rm (3)] The {\em list edge chromatic number} of $(G,\sigma)$, denoted $\chi_l'(G,\sigma)$, is the smallest positive integer $k$ such that $(G, \sigma)$ is $k$-list-edge-colorable.
 \end{itemize}
\end{definition}

For the sake of convenience, when $e_h$ is negative, write $L(e_h)$ and $\phi(e_h)$ instead of $L(h)$ and $\phi(h)$, respectively.
Observe that signed list edge coloring can be viewed as a unified generalization of both list edge coloring and signed edge coloring.
More precisely,
\begin{itemize}
\item when $\sigma=-{\bf 1}$, i.e., $\sigma(e)=-1$ for all $e\in E(G)$, every $L$-edge-coloring of $(G, \sigma)$ corresponds to an $L$-edge-coloring of its underlying graph $G$;
\item if the list edge assignment $L$ is given by
$$L(h)=M_k=\left\{
\begin{array}{ll}
 \{\pm 1, \pm 2, \dots, \pm \frac{k}{2}\} &  \mbox{ if $k$ is even};\\
 \{0, \pm 1, \pm 2, \dots, \pm \frac{k-1}{2}\} &  \mbox{ if $k$ is odd},
\end{array}
\right.
\  \forall h\in H(G),$$
then any $L$-edge-coloring of $(G,\sigma)$ coincides with a $k$-edge-coloring of $(G,\sigma)$ in the sense of Behr \cite{B2020}.
\end{itemize}
Therefore, we have $\chi'_{l}(G, -{\bf 1})=\chi'_{l}(G)$ and $\chi_{l}'(G,\sigma)\ge \chi'(G,\sigma)$,  where $\chi'(G,\sigma)$ is the edge chromatic number of $(G,\sigma)$.

In this paper, we verify the signed version of Conjecture \ref{con: list edge} for signed graphs of treewidth $3$, or of treewidth $4$ with maximum degree at least $10$, thereby improving a result of Lang \cite{Lang2016}.

\begin{theorem}\label{maintheorem}
Let $G$ be a graph of treewidth $3$, or of treewidth $4$ with $\Delta\geq10$. Then for any signature $\sigma$, $\chi'_{l}(G,\sigma)\leq\Delta+1$.
\end{theorem}

The remainder of this paper is organized as follows. In Section~\ref{s2: prelim}, we introduce the structure of graphs with given treewidth, and present a polynomial method for signed list edge coloring. Section~\ref{s3: special} investigates list edge colorings of several specific signed graphs. Section~\ref{s4: proof} completes the proof of Theorem \ref{maintheorem}.

\section{Preliminaries}\label{s2: prelim}

Let $G$ be a graph. For a vertex $v \in V(G)$, we use $d(v)=d_G(v)$ and $N(v)= N_G(v)$ to denote the degree and the neighbors, respectively, of $v$ in $G$. Let $V_\ell(G)=\{v\in V(G) : d(v)=\ell\}$ and $N(W)=\cup_{v\in W}N(v)$ for $W\subseteq V(G)$.
A {\em tree decomposition} $(T,\mathcal{B})$ of $G$ consists of a tree $T$ and a collection $\mathcal{B} =\{B_t : t\in V(T)\}$ of {\em bags} $B_t \subseteq V(G)$ such that
 \begin{itemize}
 \item $V(G) =\cup_{t\in V(T)}B_t$;
 \item for each $vw\in E(G)$ there exists a vertex $t \in V(T)$ such that $v, w\in B_t$;
\item  for each $v\in V(G)$ the subgraph induced by $\{t \in V(T) : v \in B_t\}$ is a subtree of $T$.
\end{itemize}
The {\em width} of a tree decomposition  $(T,\mathcal{B})$  is $\max_{t\in V(T)}  |B_t|-1$. The {\em treewidth} $tw(G)$ of $G$ is the minimum width over all tree decompositions of $G$.

\begin{lemma}[Han et al. \cite{H2008}]\label{lem:tw}
Let $G$ be a graph with treewidth $\ell $ and $\Delta(G)\geq \ell +1$. Then there are
 two non-empty disjoint subsets $W,U \subseteq V(G)$ and a vertex $x \notin W\cup U$ satisfying the following (see Fig. \ref{fig: treewidth}):
 \begin{itemize}
\item[\rm (1)]  $N(W) \subseteq U \cup \{x\}\cup (\cup_{i\leq \ell}V_{i}(G))$;
\item[\rm (2)] $d(x) \geq \ell+1$ and $d(w)\leq \ell$ for each $w \in W$;
\item[\rm (3)]  $W \subseteq N(x) \subseteq W \cup U$;
\item[\rm (4)]  $|U|\leq\ell$.
 \end{itemize}
\end{lemma}

\begin{figure}[htb]
\scriptsize
\captionsetup[subfloat]{labelsep=none, format=plain, labelformat=empty}
\begin{center}
\begin{tikzpicture}[scale=1.6]
{
\draw (0,0) ellipse (1cm and 0.3cm);

\draw (3,0) ellipse (1cm and 0.3cm);

\draw (3,-1) ellipse (1cm and 0.3cm);

\node [left] at (-1,0) {\tiny $U$};
\node [left] at (4.3,0) {\tiny $W$};
\node [left] at (5.8,-1) {\tiny $(\cup_{i\leq\ell}V_i(G))\setminus (U\cup W)$};
\vertex (x) at(1.5,1)[label=center:{}]{$x$};
\vertex (0) at(-0.8,0)[label=center:{}]{};
\vertex (1) at(0,0)[label=center:{}] {};
\node at (-0.4,0) {$\cdots$};
\vertex (2) at(0.8,0)[label=center:{}] {};
\node at (0.4,0) {$\cdots$};
\vertex (3) at(2.2,0)[label=center:{}] {};
\vertex (4) at(3,0)[label=center:{}] {};
\node at (2.6,0) {$\cdots$};
\vertex (5) at(3.8,0)[label=center:{}] {};
\node at (3.4,0) {$\cdots$};
\vertex (6) at(2.2,-1)[label=center:{}] {};
\vertex (7) at(3,-1)[label=center:{}] {};
\node at (2.6,-1) {$\cdots$};
\node at (3.4,-1) {$\cdots$};
\vertex (8) at(3.8,-1)[label=center:{}] {};
\draw  (0)--(x)--(2) (x)--(3) (4)--(x)--(5);
\draw(0) .. controls (0.5,-0.5)  and (2,-0.5).. (4);
\draw (2) .. controls (1.5,-0.3) and (2.3,-0.3) .. (4);
\draw  [dashed](6)--(3)--(7) (6)--(4)--(7) (4)--(8);
\draw (6)--(5)--(7) (5)--(8);

\fill[black] (-1,-0.3) circle (0.15pt);
\fill[black] (-0.7,-0.3) circle (0.15pt);
\draw (-1,-0.3) --(0)-- (-0.7,-0.3);
\fill[black] (0.6,-0.3) circle (0.15pt);
\fill[black] (1,-0.3) circle (0.15pt);
\draw (0.6,-0.3) --(2)-- (1,-0.3);

}
\end{tikzpicture}
\end{center}
\caption{Structure of the graph $G$ in Lemma \ref{lem:tw}.}
\label{fig: treewidth}
\end{figure}
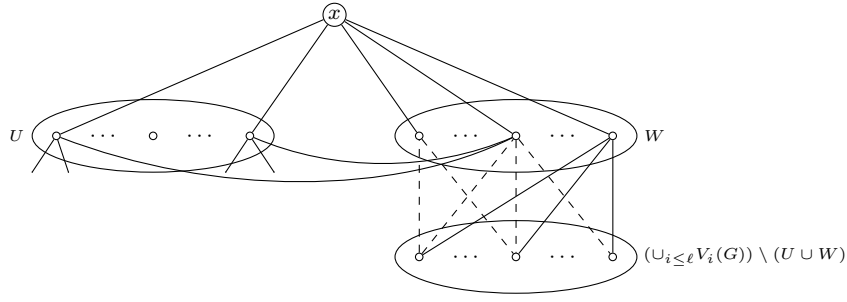

We apply Alon's celebrated Combinatorial Nullstellensatz \cite{Alon1992} to establish a sufficient condition for the existence of an $L$-edge-coloring of a signed graph.

\begin{theorem}[Combinatorial Nullstellensatz, Alon~\cite{Alon1992}]\label{Alon}
Let $\mathbb{F}$ be an arbitrary field and let $P\in \mathbb{F}[x_1, x_2, \ldots, x_m]$ be a polynomial of degree $\deg(P) = \sum_{j=1}^m i_j$, where each $i_j$ is a nonnegative integer. If the coefficient of the monomial $x_1^{i_1} x_2^{i_2} \cdots x_m^{i_m}$ in $P$ is nonzero, and if $S_1, S_2, \ldots, S_m$ are subsets of $\mathbb{F}$ with $|S_j| > i_j$ for each $j$, then there exist $s_1 \in S_1, s_2 \in S_2, \ldots, s_m \in S_m$ such that $P(s_1, s_2, \ldots, s_m) \neq 0$.
\end{theorem}

Recall that, for a half-edge $h$, $e_h$ denotes the edge containing $h$, and $\hat h$ denotes the other half-edge of $e_h$. Let $(G,\sigma)$ be a signed graph with vertex set
$V(G)=\{v_1,v_2,\dots,v_n\}$ and edge set
$E(G)=\{e_1,e_2,\dots,e_m\}$. Denote
$$\widetilde H(G)=\{h_e^{v_i}:e=v_iv_j\in E(G), i<j\}.$$
Clearly,
$$H(G)\setminus \widetilde H(G)=\{\hat h:h\in \widetilde H(G)\}.$$
For each $h\in \widetilde H(G)$, we assign an independent variable $x_h$. The variable corresponding to the other half-edge $\hat h$ of the same
edge $e_h$ is determined by the sign of $e_h$ as
\begin{equation}\label{eq: relation}
x_{\hat h}=-\sigma(e_h)x_h.
\end{equation}
Thus, if $e_h$ is negative, then $x_{\hat h}=x_h$; if $e_h$ is positive,
then $x_{\hat h}=-x_h$.
We now define the polynomial $P_{(G,\sigma)}$ as follows:
\begin{equation}\label{eq: polynomial1}
P_{(G,\sigma)}
=
P_{(G,\sigma)}(\{x_h:h\in \widetilde H(G)\})
=
\prod_{i=1}^{n}
\left(
\prod_{\substack{h,h'\in H_G(v_i)\\ 1\le I(h)<I(h')\le m}}
(x_h-x_{h'})
\right),
\end{equation}
where $I(h)=j$ if $e_h=e_j\in E(G)$. As an illustration, consider the signed graph $(G_0,\sigma_0)$ shown in Fig.~\ref{fig:cycle-polynomial}. Assign variables $x_1, x_2, \dots, x_7$ to the half-edges in
$$\widetilde{H}(G_0)=\{h_{e_1}^{1}, h_{e_2}^{2},h_{e_3}^{3},h_{e_4}^{4},h_{e_5}^{5},h_{e_6}^{1},h_{e_7}^{6}\},$$ respectively. Then, by applying Eqs.~(\ref{eq: relation}) and (\ref{eq: polynomial1}), we have
$$
P_{(G_0,\sigma_0)}=(x_1-x_6)\cdot (x_1-x_2)\cdot (-x_2-x_3)\cdot (x_3-x_4)\cdot (-x_4-x_5)\cdot (x_5+x_6)(x_5-x_7)(-x_6-x_7).
$$

\begin{figure}[htb]
\scriptsize
\captionsetup[subfloat]{labelsep=none, format=plain, labelformat=empty}
\begin{center}
\vspace{-0.35cm}
\begin{tikzpicture}
{
\vertex (1) at (180:2) [label=center:{}] {$1$};
\vertex (2) at (120:2) [label=center:{}] {$2$};
\vertex (3) at ( 60:2) [label=center:{}] {$3$};
\vertex (4) at (  0:2) [label=center:{}] {$4$};
\vertex (5) at (300:2) [label=center:{}] {$5$};
\vertex (6) at (240:2) [label=center:{}] {$6$};
\vertex (7) at (210:3.464) [label=center:{}] {$7$};

\draw[dotted,line width=0.8pt]        (1)--(2);
\draw[line width=0.8pt] (2)--(3);
\draw[dotted,line width=0.8pt]        (3)--(4);
\draw[line width=0.8pt] (4)--(5);
\draw[dotted,line width=0.8pt]        (5)--(6);
\draw[line width=0.8pt]        (7)--(6);
\draw[line width=0.8pt] (6)--(1);

\node at (135:2) {$x_1$};
\node at (165:2) {$x_1$};

\node at (105:2) {$x_2$};
\node at (75:2) {$-x_2$};

\node at (45:2) {$x_3$};
\node at (15:2) {$x_3$};

\node at (345:2) {$x_4$};
\node at (316:2.1) {$-x_4$};

\node at (285:2) {$x_5$};
\node at (255:2) {$x_5$};

\node at (224:2.06) {$-x_6$};
\node at (195:2) {$x_6$};

\node at (217:3.12) {$-x_7$};
\node at (233:2.35) {$x_7$};

\node at (150:1.5) {$e_1$};
\node at (90:1.5) {$e_2$};

\node at (30:1.5) {$e_3$};
\node at (330:1.5) {$e_4$};

\node at (270:1.5) {$e_5$};
\node at (210:1.5) {$e_6$};
\node at (218:2.6) {$e_7$};

}
\end{tikzpicture}
\vspace{-0.35cm}
\end{center}
\caption{An example of the polynomial construction for a signed graph
\((G_0,\sigma_0)\), where
\(V(G_0)=\{1,2,\ldots,7\}\),
\(E(G_0)=\{e_1,e_2,\ldots,e_7\}\), and
\(E_N(G_0)=\{e_1,e_3,e_5\}\).
Solid edges are positive and dotted edges are negative.}
\label{fig:cycle-polynomial}
\end{figure}
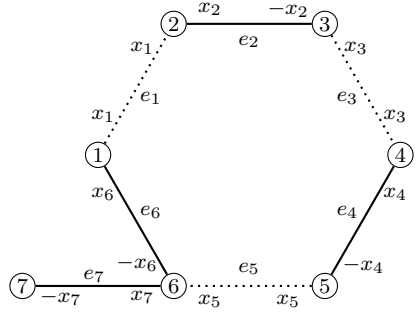

\begin{lemma}\label{CN}
Let $(G, \sigma)$ be a signed graph with $V(G) = \{v_1, v_2, \dots, v_n\}$ and $E(G) = \{e_1, e_2, \ldots, e_m\}$. Suppose the coefficient of the monomial $\prod_{h \in \widetilde{H}(G)} x_h^{i_h}$ in $P_{(G, \sigma)}$ is nonzero and $\sum_{h \in \widetilde{H}(G)} i_h = \deg(P_{(G, \sigma)})$. Then for any function $f : E(G) \to \mathbb{Z}^+$ satisfying
$$f(e_h) \geq i_h +1,\ \forall h \in \widetilde{H}(G),$$
 $(G, \sigma)$ is $f$-list-edge-colorable.
 \end{lemma}

\begin{proof}
Let $L$ be an arbitrary list edge assignment on $(G,\sigma)$ such that
$|L(h)|\ge f(e_h)$ for each $h\in \widetilde{H}(G)$. Since $|L(h)|\ge f(e_h)\ge i_h+1$ for every $h\in \widetilde{H}(G)$, it follows from Theorem~\ref{Alon} that there exist colors $a_h\in L(h)$ for all $h\in \widetilde H(G)$ such that
\begin{equation}\label{eq: color polynomial}
P_{(G,\sigma)}(\{a_h : h\in \widetilde H(G)\})\neq 0.
\end{equation}

Define a mapping $\phi$ on $H(G)$ by $\phi(h)=a_h$ and $\phi(\hat h)=-\sigma(e_h)a_h$ for each $h\in \widetilde H(G)$.
Because $L$ is a list edge assignment on $(G,\sigma)$, we have $L(\hat h)=-\sigma(e_h)L(h)$, and thus
$\phi(\hat h)\in L(\hat h)$.
Hence, by Eq.~(\ref{eq: color polynomial}), any two half-edges $h_1$ and $h_2$ incident with the same vertex receive distinct colors $\phi(h_1)\in L(h_1)$ and $\phi(h_2)\in L(h_2)$. This implies that $\phi$ is an $L$-edge-coloring of $(G,\sigma)$.
Since $L$ is arbitrary, $(G,\sigma)$ is $f$-list-edge-colorable.
\end{proof}

In a signed graph, \emph{switching} at a vertex means reversing the signs of all edges incident with it. Two signed graphs are \emph{switching equivalent} if one can be obtained from the other by a sequence of such switchings.  The following result shows that the list edge chromatic number is invariant under switching equivalence.

\begin{lemma}\label{switching}
Let $(G, \sigma)$ and $(G, \sigma')$ be switching equivalent signed graphs. For any function $f : E(G)\rightarrow \mathbb{Z}^+$, $(G, \sigma)$ is $f$-list-edge-colorable if and only if  $(G, \sigma')$ is $f$-list-edge-colorable.
Consequently, $\chi_{l}'(G,\sigma)=\chi_{l}'(G,\sigma')$.
\end{lemma}
\begin{proof}
It suffices to prove that if $(G,\sigma)$ is $f$-list-edge-colorable, then so is $(G,\sigma')$.
Let $L'$ be a list edge assignment on $(G,\sigma')$ such that $|L'(h)| \geq f(e_h)$ for every $h\in H(G)$. Since $(G,\sigma)$ and $(G,\sigma')$ are switching equivalent, there exists a subset $X\subseteq V (G)$ such that $(G,\sigma)$ is obtained from $(G,\sigma')$ by switching at all vertices in $X$. Define a list edge
assignment $L$ on $(G,\sigma)$ by
$$L(h)=\left\{
\begin{array}{rl}
-L'(h) & \mbox{ if the end of $h$ belongs to $X$};\\
L'(h) & \mbox{ otherwise},
\end{array}
\right. \forall h\in H(G).
$$
Then $|L(h)| = |L'(h)| \geq f(e_h)$ for all $h \in H(G)$. Since $(G, \sigma)$ is $f$-list-edge-colorable, there exists an $L$-edge-coloring $\phi$ of $(G,\sigma)$.  Define a coloring $\phi'$ for $(G,\sigma')$ by
$$
\phi'(h)=\left\{
\begin{array}{rl}
-\phi(h) & \mbox{ if the end of $h$ belongs to $X$};\\
\phi(h) & \mbox{ otherwise},
\end{array}
\right. \forall h\in H(G).
$$
It is straightforward to verify that $\phi'$ is an $L'$-edge-coloring of $(G, \sigma')$. Therefore, $(G, \sigma')$ is $f$-list-edge-colorable, which completes the proof.
\end{proof}


\section{List edge colorings of specific signed graphs}
\label{s3: special}
Before proceeding to the proof of Theorem \ref{maintheorem}, we  investigate list edge colorings of several classes of signed graphs. For two integers $n_1, n_2$ with $n_1\leq n_2$, let $[n_1, n_2]$ be the set of integers between $n_1$ and $n_2$.

\begin{lemma}\label{le: cycle-1}
Let $(C,\sigma)$ be a signed cycle, and let $f : E(C)\to \mathbb{Z}^+ $ be a function. Then $(C,\sigma)$ is $f$-list-edge-colorable if one of the following conditions holds:
\begin{itemize}
\item [\rm (1)] There exist two edges $e', e'' \in E(C)$ such that $f(e')=1$, $f(e'')\geq 3$,  and  $f(e)\geq 2$ for each $e\in E(C)\setminus\{e',e''\}$;

\item [\rm (2)] $(C,\sigma)$ is balanced and $f(e)\geq2$ for every $e\in E(C)$.
\end{itemize}
\end{lemma}

\begin{proof}
Let $C=12\cdots n1 $ be the cycle with edges $e_i = i(i+1)$ for $i\in [1,n-1]$ and $e_n = n1$.
By Lemma \ref{switching}, we may assume that $\sigma(e_i)=-1$ for $i\in [1, n-1]$ and $\sigma(e_n)\in \{-1,1\}$.
Assign a variable $x_i$ to the half-edge of $e_i$ in $\widetilde{H}(C)$ for each $i \in [1, n]$. By Eq.~(\ref{eq: polynomial1}),
$$
P_{(C, \sigma)} = (x_1 - x_n)(x_1 - x_2)(x_2 - x_3) \cdots (x_{n-2} - x_{n-1})(x_{n-1}+ \sigma(e_n) x_n).
$$

(1)  Without loss of generality, assume that $e'=e_1$ and $e''=e_i$ for some $i\in [2,n-1]$.
Then  the monomial $x_2 \cdots x_{i-1} x^2_i x_{i+1}\cdots x_n$ in $P_{(C,\sigma)}$ has coefficient $(-1)^i$, and  is nonzero. For this monomial, the exponent of $x_1$ is $0$, the exponent of
$x_i$ is $2$, and the exponent of every other variable is $1$.
Therefore, the assumptions that
$f(e_1)=1$, $f(e_i)\ge 3$, and $f(e)\ge 2$
for all $e\in E(C)\setminus\{e_1,e_i\}$ imply that
$f(e_h)\ge i_h+1$ for every $ h\in \widetilde H(C)$.
Thus the hypotheses of Lemma~\ref{CN} are satisfied for this monomial. Hence $(C,\sigma)$ is $f$-list-edge-colorable.

(2) Since $(C,\sigma)$ is balanced,  $n$ is even if $\sigma(e_n)=-1$, and $n$ is odd if $\sigma(e_n)=1$. In  the former case, $\chi_l'(C, \sigma)=\chi_l'(C, -{\bf 1})=\chi_l'(C)=2$.
In the latter case, the coefficient of the monomial $x_1 x_2 \cdots x_n$ in $P_{(C, \sigma)}$ equals $-2$. 
Since $f(e)\ge 2$ for every $e\in E(C)$, the
hypothesis $f(e_h)\ge i_h+1$ of Lemma~\ref{CN} is satisfied for this monomial.
Thus, in both cases, $(C,\sigma)$ is $f$-list-edge-colorable.
\end{proof}

\begin{lemma}\label{le: cycle-2}
Let $C$ be an even cycle with a signature $\sigma$, and $L$ be a list edge assignment on $(C,\sigma)$ satisfying $|L(h)|=2$ for every $h\in H(C)$.
Then $(C, \sigma)$ admits no $L$-edge-coloring if and only if it is unbalanced and there exists some $a\in \mathbb{Z}$ such that $L(h)=\{a,-a\}$ for all $h\in H(C)$.
\end{lemma}

\begin{proof}
Let $C=12\cdots n1 $ be the cycle with edges $e_i = i(i+1)$ for $i\in [1,n-1]$ and $e_n = n1$.
 By Lemma \ref{switching}, we may assume $\sigma(e_i) = -1$ for all $i \in [1, n - 1]$. For each negative edge $e_i$, $i\in[1,n-1]$, its two half-edge lists are
the same, since
$L(\hat h)=-\sigma(e_i)L(h)=L(h)$.
We denote this common list by $L(e_i)$.

Sufficiency.  Since $n$ is even and $(C,\sigma)$ is unbalanced, we have $\sigma(e_n) = 1$. A direct verification shows that $(C, \sigma)$ admits no $L$-edge-coloring since $L(h)=\{a,-a\}$ for all $h\in H(C)$.

Necessity. By Lemma \ref{le: cycle-1}-(2), $(C, \sigma)$ is unbalanced, and so $\sigma(e_n)=1$ since $n$ is even. We first show that $L(e_i)=L(e_{i+1})$ for every $i\in [1,n-2]$. Suppose otherwise, then there  exists some $j\in[1,n-2]$ such that $L(e_{j+1}) \setminus L(e_j)\neq\emptyset$. Choose a color $a\in L(e_{j+1}) \setminus L(e_j)$ and assign it to  $e_{j+1}$. Now the remaining sequence $e_{j+2},\dots,e_{n-1}$, $h_{e_n}^n$, $h_{e_n}^1$, $e_1,\dots,e_j$ can be colored greedily, which gives an $L$-edge-coloring of $(C, \sigma)$, a contradiction. Thus there exist colors $a, b$ such that $L(e_i)=\{a,b\}$ for all $i\in [1,n-1]$.
Because $n$ is even, the path $C-e_n$ admits a proper edge coloring $\phi$ in which both $e_1$ and $e_{n-1}$ receive the color $a$. If $L(h^1_{e_n})\neq \{a,-a\}$,
then there exists a color $c\in L(h^1_{e_n})\setminus\{a,-a\}$.
Assigning $c$ to $h^1_{e_n}$ and $-c$ to $h^n_{e_n}$ extends $\phi$ to an $L$-edge-coloring, a contradiction.
Hence $L(h^1_{e_n})=L(h^n_{e_n})=\{a,-a\}$. By symmetry between $a$ and $b$, we also have $L(h^1_{e_n})=L(h^n_{e_n})=\{b,-b\}$, which forces $b=-a$. Therefore the claim holds.
\end{proof}


\begin{lemma}\label{le: cycle+e0}
Let $C = 12 \cdots n1$ with $n\ge 4$ be a cycle and let $G=C+e_0$, where $e_0$ is either a pendant edge $10$ or a chord $1t$ with $t \in [3, n - 1]$. Define a function $f: E(G)\to \mathbb{Z}^+$ by
$$f(e) = d_G(i) + d_G(j) - 2, \forall e=ij\in E(G).$$
 Then $(G, \sigma)$ is $f$-list-edge-colorable for any signature $\sigma$.
\end{lemma}
\begin{proof}
Denote $e_i = i(i+1)$ for $i \in [1, n-1]$ and $e_n = n1$. By Lemma \ref{switching}, we may assume that $\sigma(e_i) = -1$ for $i \in [1, n-1]$ and also $\sigma(e_0)=-1$ when $e_0=10$. Associate the half-edge of $e_i$ in $\widetilde{H}(G)$ with a variable $x_i$ for each $i \in [0, n]$.  Using (\ref{eq: polynomial1}), we obtain

$$
P_{(G, \sigma)} =
\begin{cases}
P & \text{if } e_0 = 10; \\
P \cdot (-\sigma(e_0)x_0 - x_{t-1})(-\sigma(e_0)x_0 - x_t) & \text{if } e_0 = 1t,
\end{cases}
$$
where $P = (x_0 - x_n)(x_1 - x_n)(x_{n-1} + \sigma(e_n)x_n) \prod_{i=0}^{n-2}(x_i - x_{i+1})$.

If $e_0 = 10$, then $\deg(P_{(G, \sigma)}) = n + 2$ and the coefficient of $x_0x_1 \cdots x_{n-1}x_n^2$ in $P_{(G, \sigma)}$ is equal to $1$. 
Since $f(e_1)=f(e_n)=3$ and $f(e_i)=2$ for
$i\in\{0\}\cup[2,n-1]$, we have $f(e_h)\ge i_h+1$ for this monomial. Hence Lemma~\ref{CN} implies that $(G,\sigma)$ is $f$-list-edge-colorable.

If $e_0 = 1t$, then $\deg(P_{(G, \sigma)}) = n + 4$ and the coefficient of $x_0^3x_1 \cdots x_{n-1}x_n^2$ in $P_{(G, \sigma)}$ is also equal to $1$. Since $f(e_0) = 4$, $f(e_i) = 3$ for $i \in \{1, t - 1, t, n\}$ and $f(e_i) = 2$ for $i \in [2, n - 1] \setminus \{t - 1, t\}$, 
we have $f(e_h)\ge i_h+1$ for this monomial.
Hence Lemma~\ref{CN} implies that $(G,\sigma)$ is $f$-list-edge-colorable.
\end{proof}

\begin{lemma}\label{le: K33}
Let $G$ be the complete bipartite graph $K_{3,3}$. Then  for any signature $\sigma$, $(G,\sigma)$ is $3$-list-edge-colorable.
\end{lemma}

\begin{proof}
Suppose for contradiction that there exist a signature $\sigma$ and a list edge assignment $L$ on $(G,\sigma)$ with $|L(h)|=3$ for every $h\in H(G)$, such that $(G,\sigma)$ admits no $L$-edge-coloring. Let $V(G)=[1,6]$ with bipartition $A=\{1,2,3\}$ and $B=\{4,5,6\}$, and $E(G)=\{e_1, \dots, e_9\}$, where $e_1=14,e_2=15,e_3=16,e_4=24,e_5=25,e_6=26,e_7=34,e_8=35$ and $e_9=36$. Assign a variable $x_i$ to the half-edge of $e_i$ in $\widetilde{H}(G)$ for each $i \in [1,9]$.

\begin{claim}\label{cl: one positive}
$(G,\sigma)$ is switching equivalent to a signed graph $(G,\sigma')$ with exactly one positive edge.
\end{claim}

\noindent {\it Proof of Claim \ref{cl: one positive}.} Since $G=K_{3,3}$, $(G,\sigma)$ contains a balanced Hamilton cycle; without loss of generality, take the cycle formed by the edges in $E(G)\setminus \{e_1, e_5, e_9\}$.
By Lemma \ref{switching}, we may assume that $\sigma(e)=-1$ for every $e\in E(G)\setminus \{e_1, e_5, e_9\}$.
Using (\ref{eq: polynomial1}), we obtain
\begin{equation*}
\begin{split}
P_{(G, \sigma)}=&(x_1-x_2)(x_1-x_3)(x_2-x_3)\cdot(x_4-x_5)(x_4-x_6)(x_5-x_6)\cdot\\
&(x_7-x_8)(x_7-x_9)(x_8-x_9)\cdot(-\sigma(e_1)x_1-x_4)(-\sigma(e_1)x_1-x_7)(x_4-x_7)\cdot\\
&(x_2+\sigma(e_5)x_5)(x_2-x_8)(-\sigma(e_5)x_5-x_8)\cdot
(x_3-x_6)(x_3+\sigma(e_9)x_9)(x_6+\sigma(e_9)x_9).
\end{split}
\end{equation*}
A direct calculation shows that the coefficient of $x_1^2x_2^2\cdots x_9^2$ in $P_{(G,\sigma)}$ equals
$$
\eta_1=4+2(\sigma(e_1)+\sigma(e_5)+\sigma(e_9))-2\sigma(e_1)\sigma(e_5)\sigma(e_9).
$$
If at least two of $\{e_1, e_5, e_9\}$ are positive, then $\eta_1\neq 0$. Applying Lemma~\ref{CN} to the monomial $x_1^2x_2^2\cdots x_9^2$ with $f(e)=3$ for all $e\in E(G)$, we obtain that $(G,\sigma)$ is
$3$-list-edge-colorable, a contradiction. 
If none of them is positive, then $\chi_l'(G,\sigma)=\chi_l'(G,-{\bf 1})=\chi_l'(G)=3$, contradicting our assumption. Hence exactly one of $\{e_1, e_5, e_9\}$ is positive, proving the claim. $\Box$

\medskip

By Claim \ref{cl: one positive} and Lemma \ref{switching}, we assume that $e_9$ is the unique positive edge in $(G,\sigma)$. Let $h_9$ and $\hat{h}_9$ be the half-edges of $e_9$ incident with vertices $3$ and $6$, respectively. Then $L(h_9)=-L(\hat{h}_9)$ and $L(e)=L(h_{e}^i)=L(h_e^j)$ for every $e=ij\in E(G)\setminus \{e_9\}$.

\begin{claim}\label{cl: four edges}
For any $e\in \{e_1, e_2,e_4, e_5\}$,  every color of $L(e)$ belongs to the lists of at least three edges adjacent to $e$.
\end{claim}

\noindent {\it Proof of Claim \ref{cl: four edges}.} Assign a color $a\in L(e_1)$ to $e_1$, and let $L'$ be the list edge assignment on $(G-e_1, \sigma)$ obtained from $L$ by setting $L(e)$ to $L'(e)=L(e)\setminus \{a\}$ for each $e\in \{e_2, e_3, e_4,e_7\}$. Then $(G-e_1, \sigma)$ admits no $L'$-edge-coloring. By (\ref{eq: polynomial1}),
\begin{equation*}
\begin{split}
P_{(G-e_1, \sigma)}= &(x_2-x_3)\cdot(x_4-x_5)(x_4-x_6)(x_5-x_6)\cdot(x_7-x_8)(x_7-x_9)(x_8-x_9)\\
& (x_4-x_7)\cdot(x_2-x_5)(x_2-x_8)(x_5-x_8) \cdot (x_3-x_6)(x_3+x_9)(x_6+x_9).
\end{split}
\end{equation*}
One can easily check that the coefficients of the monomials $(x_2^2\cdots x_9^2)/(x_2x_3)$,  $(x_2^2\cdots x_9^2)\\/(x_2x_4)$, $(x_2^2\cdots x_9^2)/(x_2x_7)$ and $(x_2^2\cdots x_9^2)/(x_3x_7)$ in $P_{(G-e_1, \sigma)}$ are respectively $\eta_2=2, \ \eta_3=1, \ \eta_4=-3$ and $\eta_5=1$.
If the color $a$ appears in the lists of at most two edges among $\{e_2, e_3, e_4, e_7\}$,
then the remaining list sizes match one of the above
nonzero monomials. Hence Lemma~\ref{CN} implies that $(G-e_1,\sigma)$ admits an $L'$-edge-coloring, a contradiction.
Therefore, by the symmetry among $\{e_1, e_2, e_4, e_5\}$ and the arbitrariness of $a$, the claim follows.
$\Box$

\begin{claim}\label{cl: monchromatic}
There exists a perfect matching $M$ of $G-e_9$ such that $\cap_{e\in M}L(e)\neq \emptyset$.
\end{claim}

\noindent {\it Proof of Claim \ref{cl: monchromatic}.} Take a color $a\in L(e_1)$, and suppose first that $a\notin L(e_3)$. By Claim \ref{cl: four edges}, we have $a\in L(e_2)\cap L(e_4)\cap L(e_7)$, and so $a\in L(e_5)\cap L(e_8)$. If $a\in L(e_6)$, then $\{e_1, e_6, e_8\}$ is a required perfect matching. If $a\notin L(e_6)$, then assign color $a$ to both $e_1$ and $e_5$, and consider the list edge assignment $L'$ on $(G-e_1-e_5, \sigma)$ obtained from $L$ by removing $a$ from the list of every edge incident with $e_1$ or $e_5$. Using an argument similar to the proof of Claim \ref{cl: four edges}, we find that $(G-e_1-e_5, \sigma)$ admits an $L'$-edge-coloring, a contradiction. Hence $a\in L(e_3)$ must hold. By the arbitrariness of $a$ and the symmetry between $e_3$ and $e_7$, we obtain that $L(e_1)=L(e_3)=L(e_7)$.

If $L(e_5)\cap L(e_1)\neq \emptyset$, then $\{e_3, e_5,e_7\}$ is a perfect matching satisfying the claim. Assume therefore that $L(e_5)\cap L(e_1)=\emptyset$. Choose $b\in L(e_5)$. Then $b\notin L(e_1)$ and by Claim \ref{cl: four edges}, it belongs to  at least one of $\{L(e_2), L(e_4)\}$, say $b\in L(e_2)$. Since $L(e_1)=L(e_3)$, $b$ appears in the lists of at most two edges adjacent to $e_2$, contradicting Claim \ref{cl: four edges}. Thus the claim is proved.
$\Box$
\medskip

\noindent {\bf The final step.} By Claim \ref{cl: monchromatic}, let $M$ be a perfect matching of $G-e_9$ and choose a  color $a\in \cap_{e\in M}L(e)$.
Define a list edge assignment $L'$ on $(G-M,\sigma)$ by setting $L'(e)=L(e)\setminus \{a\}$ for $e\in E(G-M-e_9)$ and $L'(h)=L(h)\setminus \{a, -a\}$ for $h\in \{h_9, \hat{h}_9\}$. Since $(G,\sigma)$ admits no $L$-edge-coloring, $(G-M,\sigma)$ admits no $L'$-edge-coloring. Because $G-M$ is an unbalanced even cycle, Lemma \ref{le: cycle-1}-(1) and \ref{le: cycle-2} imply that $a\in L(e)$ for every $e\in E(G-e_9)$, and one of the following holds:
\begin{itemize}
\item there exists a color $b\notin \{a, -a\}$ such that $L'(h)=\{b, -b\}$ for $h\in H(G-M)$;
\item $\{a, -a\}\subseteq L(h_9)$.
\end{itemize}

Suppose the former holds. Then $L(e)=L(h_9)=\{a, b, -b\}$ for $e\in E(G-M-e_9)$ and $L(\hat{h}_9)=\{-a, b, -b\}$. If $L(e)=\{a, b, -b\}$ for each $e\in M$,  we can define an $L$-edge-coloring $\phi$ of $(G,\sigma)$ by $\phi(e_1)=\phi(e_6)=\phi(e_8)=b$, $\phi(e_3)=\phi(e_5)=\phi(e_7)=-b$, $\phi(e_2)=\phi(e_4)=\phi(h_9)=-\phi(\hat{h}_9)=a$, a contradiction.  If $L(e')\neq \{a, b, -b\}$ for some $e'\in M$, let $M'$ be the unique perfect matching of $G-M-e_9$. Since $\{a, -a\}\not\subseteq L(h_9)$, applying the same argument to $M$ yields $L(h_9)=L(e')$, again a contradiction.

Suppose now that $\{a, -a\}\subseteq L(h_9)$. Write $L(h_9)=\{a,-a, c\}$ for some color $c$. Then $L(\hat{h}_9)=\{a,-a,-c\}$.
Let $M_1=\{e_1, e_5, e_9\}$ and assign color $a$ to each of $\{e_1, e_5, h_9\}$ and $-a$ to $\hat{h}_9$. Since $(G,\sigma)$ admits no $L$-edge-coloring, $(G-M_1, \sigma)$ admits no $L_1$-edge-coloring, where $L_1(e_i)=L(e_i)\setminus \{a\}$ $(i\in  \{2, 4, 7,8\})$ and $L_1(e_i)=L(e_i)\setminus \{a,-a\}$ $(i\in  \{3,6\})$. By Lemma \ref{le: cycle-1}-(2), one of $\{e_3, e_6\}$, say $e_6$, satisfies $\{a, -a\}\subseteq L(e_6)$. By symmetry, one of $\{e_7, e_8\}$, say $e_8$,  also satisfies $\{a, -a\}\subseteq L(e_8)$.
Now define an edge coloring $\psi$ of $(G-e_1, \sigma)$ by $\psi(e_3)=\psi(e_5)=\psi(e_7)=a$, $\psi(e_6)=\psi(e_8)=-a$ and $\psi(z)\in L(z)\setminus \{a, -a\}$ for $z\in \{e_2, e_4, h_9, \hat{h}_9\}$. Because $(G,\sigma)$ admits no $L$-edge-coloring, we must have $L(e_1)=\{a, \psi(e_2), \psi(e_4)\}$ and $-a\in L(e_2)\cap L(e_4)$.
By Claim \ref{cl: four edges}, $-a\in L(e_3)\cap L(e_5)\cap L(e_7)$  since $-a\notin L(e_1)$. Consequently, $(G,\sigma)$ admits an $L$-edge-coloring $\phi$ defined by $\phi(e_3)=\phi(e_4)=\phi(e_8)=a$, $\phi(e_2)=\phi(e_6)=\phi(e_7)=-a$ and $\phi(z)\in L(z)\setminus \{a, -a\}$ for $z\in \{e_1, e_5, h_9, \hat{h}_9\}$, which is a contradiction.

The proof of the lemma is complete.
\end{proof}


\section{Proof of Theorem \ref{maintheorem}} \label{s4: proof}
Before giving the detailed proof of Theorem~\ref{maintheorem}, we first outline the main steps. Let $(G,\sigma)$ be a minimum counterexample to the theorem. Using the
reducibility results in Section~\ref{s3: special}, we show that $(G,\sigma)$ contains neither an edge with small degree-sum nor an alternating cycle (see Claims \ref{cl: sumdegree} and \ref{alcycle1}). Next, by the structural lemma for
bounded treewidth graphs, we find a suitable bipartite subgraph in  $G,$ that  contains a copy of $K_{3,3}$. However, this copy is reducible by Lemma~\ref{le: K33}, which contradicts the minimality of $(G,\sigma)$.

We now proceed with the proof of Theorem~\ref{maintheorem}.

\noindent {\bf Proof of Theorem \ref{maintheorem}.}  Suppose, for contradiction, that $(G, \sigma)$ is a counterexample to Theorem \ref{maintheorem} with minimum $|E(G)|$. Then $\Delta\ge 3$ by Lemma \ref{le: cycle-1}-(1) and there exists a list edge assignment $L$ on $(G,\sigma)$ such that $|L(h)|=\Delta+1$ for each $h\in H(G)$ and $(G,\sigma)$ admits no $L$-edge-coloring.

For a subgraph $Q$ of $G$, let $\bar{Q}=G-E(Q)$ and define a function $f_Q: E(Q)\to \mathbb{Z}$ by
\begin{equation}\label{eq: fQ}
f_Q(e)=\Delta+1-(d_{\bar{Q}}(u)+d_{\bar{Q}}(v)), \forall e=uv\in E(Q).
\end{equation}
\setcounter{claim}{0}
\begin{claim}\label{cl: delete}
Let $Q$ be a proper subgraph of $G$. Then either $(Q,\sigma)$ is not $f_Q$-list-edge-colorable, or $tw(G)=4$ and $\Delta(G-e)=9$ for any $e\in E(Q)$.
\end{claim}

\noindent {\it Proof of Claim \ref{cl: delete}.} Suppose to the contrary that $(Q,\sigma)$ is $f_Q$-list-edge-colorable and that either $tw(G)=3$ or $tw(G)=4$ and there exists some edge $e_0\in E(Q)$ such that $\Delta(G-e_0)\ge 10$. If $tw(G)=3$, we choose an arbitrary edge $e_0\in E(Q)$; if $tw(G)=4$, we choose such an edge $e_0$ with $\Delta(G-e_0)\ge 10$. In both cases, $G-e_0$ still satisfies the hypothesis of Theorem~\ref{maintheorem}. Hence, by the minimality of $G$, $(G-e_0,\sigma)$ admits an $L|_{H(G-e_0)}$-edge-coloring $\phi'$. Uncoloring every half-edge in $H(Q-e_0)$ from $\phi'$, we obtain an $L|_{H(\bar{Q})}$-edge-coloring $\phi$ of $(\bar{Q},\sigma)$.  Define a mapping $L_{Q}: H(Q) \to 2^{\mathbb{Z}}$ by
$$
L_{Q}(h)=L(h)\setminus \left(\{\phi(h) : h\in H_{\bar{Q}}(v)\}\cup \{-\sigma(e_h)\phi(h) : h\in H_{\bar{Q}}(u)\}\right),\ \forall h\in H(Q),
$$
where $v$ is the end of $h$ and $u$ is another end of the edge $e_h$. Then $L_Q$ is a list edge assignment on $(Q,\sigma)$ with $|L_Q(h)|\ge f_Q(e_h)$ for every $h\in H(Q)$. Since $(Q,\sigma)$ is $f_Q$-list-edge-colorable, it admits an $L_Q$-edge-coloring. Combining this coloring with $\phi$ yields an $L$-edge-coloring of the entire signed graph $(G, \sigma)$, a contradiction. Hence the claim holds.
\hfill $\Box$

\medskip

Let $S=\{uv\in E(G) : d_G(u)+d_G(v)\leq \Delta+2\}$.

\begin{claim}\label{cl: sumdegree}
 $|S|\leq 1$. Moreover, if $S=\{uv\}$ with $d_G(u)\ge d_G(v)$, then $tw(G)=4$, $\Delta=10$, $V_{\Delta}(G)=\{u\}$ and $V_1(G)=\{v\}$.
\end{claim}

\noindent {\it Proof of Claim \ref{cl: sumdegree}.}
Suppose there exists an edge $uv\in E(G)$ such that $d_G(u)+d_G(v)\leq \Delta+2$. Choose such an edge $uv$ with  $d_G(u)+d_G(v)$ minimum, and let $Q=uv$. Then $(Q,\sigma)$ is $f_Q$-list-edge-colorable because  $f_Q(uv)=(\Delta+1)-(d_{\bar{Q}}(u)+d_{\bar{Q}}(v))=(\Delta+1)-(d_G(u)+d_G(v)-2)\ge 1$.
By Claim \ref{cl: delete}, we have $tw(G)=4$ and $\Delta(G-uv)=9$. Then $\Delta=\Delta(G-uv)+1=10$. Without loss of generality, assume that $d_G(u)\ge d_G(v)$. It follows that $V_{\Delta}(G)=\{u\}$ and $d_G(v)\leq 2$. If $d_G(v)=2$, then by the choice of $uv$, the other neighbor of $v$ would also have degree $\Delta$, a contradiction. Thus $V_1(G)=\{v\}$ and $S=\{uv\}$ satisfies the required conditions. This completes the proof.
\hfill$\Box$

\medskip

A cycle $C=v_1v_2\cdots v_{2k}$ is called an \textit{alternating cycle} if $\{v_1, v_3, \dots, v_{2k-1}\}\subseteq V_3(G)$ and $\{v_2, v_4, \dots, v_{2k}\}\subseteq V_\Delta(G)$.

\begin{claim}\label{alcycle1}
There is no alternating cycle in $G$.
\end{claim}
\noindent {\it Proof of Claim \ref{alcycle1}.}
Suppose to the contrary that $G$ contains an alternating cycle $C=v_1v_2\cdots v_{2k}$ with $\{v_1, v_3, \dots, v_{2k-1}\}\subseteq V_3(G)$ and $\{v_2, v_4, \dots, v_{2k}\}\subseteq V_\Delta(G)$. Then $|V_{\Delta}(G)|\ge k\ge 2$ and  by Claim \ref{cl: sumdegree}, we must have  $S=\emptyset$.
Let $u$ be the neighbor of $v_1$ distinct from $v_2$ and $v_{2k}$, and set $Q=C+v_1u$.
Then $u\in V_\Delta(G)$ since $S=\emptyset$, and $\Delta(G-e)=\Delta$ for any $e\in E(Q)$. 
Whether $u\notin V(C)$ or $u\in V(C)$, the graph $Q=C+v_1u$ is a cycle with
one additional pendant edge or chord. Moreover, a direct computation using
Eq.~~(\ref{eq: fQ}) shows that
$f_Q(e)=d_Q(x)+d_Q(y)-2$ for every edge $e=xy\in E(Q)$.
Thus $Q$ satisfies the hypotheses of Lemma~\ref{le: cycle+e0}.
Hence $(Q,\sigma)$ is $f_Q$-list-edge-colorable, contradicting Claim~\ref{cl: delete}.\hfill $\Box$

\begin{claim}\label{cl: D>3}
 $\Delta\ge 4$.
 \end{claim}

\noindent {\it Proof of Claim \ref{cl: D>3}.} Suppose $\Delta\leq 3$. By Claim \ref{cl: sumdegree}, $G$ must be cubic.
Let $P=v_1v_2\cdots v_k$ be a longest path in $G$.
Then all neighbors of $v_k$ belong to $\{v_1, \dots, v_{k-1}\}$. Because $d_G(v_k)=3$, $v_k$ has two neighbors in $\{v_1, \dots, v_{k-2}\}$, denoted by $v_i, v_j$ with $i<j$. Hence at least one of the three cycles $v_iPv_jv_kv_i$, $v_iPv_kv_i$ and $v_jPv_kv_j$ is balanced. Denote this balanced cycle by $Q$.  
For every $e\in E(Q)$, Eq.~(\ref{eq: fQ}) gives $f_Q(e)=(\Delta+1)-(3-2)-(3-2)=2$.
Since $Q$ is a balanced cycle, condition (2) of Lemma~\ref{le: cycle-1}
is satisfied. Hence $(Q,\sigma)$ is $f_Q$-list-edge-colorable,
contradicting Claim~\ref{cl: delete}. Therefore $\Delta\ge 4$.
\hfill$\Box$

\medskip

Let $\ell=tw(G)$. Note that $\Delta\ge 4=\ell+1$ by Claim \ref{cl: D>3} when $\ell=3$, and $\Delta\ge 10>\ell+1$ when $\ell=4$.
By Lemma \ref{lem:tw}, there exist two non-empty disjoint subsets $U=\{u_1,\dots, u_s\}$ and $W=\{w_1, \dots, w_t\}$ of $V(G)$ and a vertex $x \notin W \cup U $ satisfying:
\begin{itemize}
\item[\rm (1)]  $N(W) \subseteq U \cup \{x\}\cup \cup_{i\le \ell}V_i(G)$;
\item[\rm (2)] $d_G(x) \geq \ell+1$ and $d_G(w)\leq \ell$ for each $w \in W$;
\item[\rm (3)]  $W \subseteq N(x) \subseteq W \cup U$;
\item[\rm (4)]  $s\leq\ell$.
 \end{itemize}

\begin{claim}\label{cl: tw=4}
$\ell=4$.
 \end{claim}

\noindent {\it Proof of Claim \ref{cl: tw=4}.} Suppose to the contrary that $\ell=3$. From (2) and Claim \ref{cl: sumdegree}, we obtain that $W\subseteq V_3(G)$ and $d_G(x)=\Delta$.
By (3) and (4), we have $s+t\ge d_G(x)=\Delta\ge 4$ and $s\leq 3$. Note that $V_{3}(G)$ is an independent set of $G$.

If $t\geq2$, then $w_1$ and $w_2$ have a common neighbor $u\in U$, and so $xw_1uw_2x$ is an alternating cycle,  contradicting Claim \ref{alcycle1}. Thus, $t=1$.
Consequently, $s=3$ and $\Delta=d_G(x)=4$, moreover, $N(x)=U\cup W$. Without loss of generality, assume that $u_1, u_2$ are the neighbors of $w_1$. Let $Q=xu_1w_1u_2x+xw_1$. Then $Q$ is a cycle with one additional chord.
By Eq. (\ref{eq: fQ}), we have
$$
f_Q(e)= \left\{
\begin{array}{ll}
(\Delta+1)-(d_G(x)-3)-(d_G(w_1)-3)=4 & \mbox{if $e=xw_1$};\\
(\Delta+1)-(d_G(x)-3)-(\Delta-2)=2 & \mbox{if $e\in \{xu_1, xu_2\}$};\\
(\Delta+1)-(d_G(w_1)-3)-(\Delta-2)=3 & \mbox{if $e\in \{w_1u_1, w_1u_2\}$}.
\end{array}
\right.
$$
Thus the values of $f_Q$ coincide with those required in Lemma~\ref{le: cycle+e0}.
Hence Lemma~\ref{le: cycle+e0} implies that $(Q,\sigma)$ is
$f_Q$-list-edge-colorable, contradicting Claim~\ref{cl: delete}. This completes the proof. \hfill$\Box$

\medskip

Recall that $\Delta\ge 10$ by Claim \ref{cl: tw=4} and the hypothesis of the theorem. Let $W_3 = W\cap V_3(G)$ and $W_4 = W\cap V_4(G)$.

\begin{claim}\label{nw}
$|W_3|\leq2$.
\end{claim}

\noindent {\it Proof of Claim \ref{nw}.}
Suppose, for contradiction, that $|W_3|\geq3$. Let  $\{w_1,w_2,w_3\}\subseteq W_3$. Since $N(W)\subseteq U\cup \{x\}$ and $|U|\leq4$, by the pigeonhole principle, at least two of $\{w_1, w_2, w_3\}$, say $w_1, w_2$, share a common neighbor $u\in U$. Then $xw_1uw_2x$ forms an alternating cycle, contradicting Claim \ref{alcycle1}. Hence $|W_3|\leq2$.
\hfill$\Box$
\medskip

For disjoint subsets $X, Y\subseteq V(G)$, let $E(X,Y)$ denote the set of edges of $G$ with one end in $X$ and the other end in $Y$.

\begin{claim}\label{K}
The bipartite graph $B$ induced by the edge set $E(W, \{x\} \cup U)$ contains a $K_{3,3}$ as a subgraph.
\end{claim}

\noindent {\it Proof of Claim \ref{K}.} Assume first that $|W_4|\ge 5$. Choose a subset  $W_4'\subseteq W_4$ with $|W_4'|=5$. Since $|E(U, W_4')|=3|W_4'|=15$ and $|U| \leq 4$,  there exist two vertices of $U$, say $u_1, u_2$, each adjacent to at least four vertices in $W_4'$.  Let $P=N(u_1)\cap W_4'$ and $Q=N(u_2)\cap W_4'$. Then $|P| \geq 4$ and $|Q| \geq 4$. Since $|W_4'| =5$, we have $|P \cap Q| \geq |P| + |Q| - |W_4'|\ge 4 + 4 - 5 = 3$. Consequently, the subgraph of $B$ induced by $E(P\cap Q, \{x, u_1, u_2\})$ contains a copy of $K_{3,3}$.

Now assume that $|W_4|\leq 4$. If $W\ne W_3\cup W_4$, then there exists a vertex $w\in W$ with $d_G(w)\le 2$, since every vertex in $W$ has degree at most $4$. Since $xw\in E(G)$ and $d_G(x)\le \Delta$, we have $d_G(x)+d_G(w)\le \Delta+2.$
Thus $xw\in S$. By Claim~\ref{cl: sumdegree}, this implies that $\Delta=10$, $V_\Delta(G)=\{x\}$, and $V_1(G)=\{w\}$.
Then all remaining neighbors in $W$ have degree $4$. Indeed, if there is a vertex $w'\in W$ with $2\leq d_G(w')\leq3$, then $w'$ has at least one neighbor other than $x$, say $y$.
Since $V_\Delta(G)=\{x\}$, we have $d_G(y)\leq\Delta-1$. Thus $w'y\in S$, which contradicts Claim~\ref{cl: sumdegree}, since $|S|\le 1$. Consequently, $|W_4|\ge d_G(x)-|U|-|V_1(G)|\ge 10-4-1=5,$ contradicting the assumption $|W_4|\le 4$. Hence $W=W_3\cup W_4$.

Combining conditions (2)-(4) with Claim~\ref{cl: sumdegree},
$d_G(x)\geq\Delta-1\geq9$; moreover, $d_G(x)=\Delta\geq10$ if $|W_3|\neq0$. We claim that $|W_3|=2$. If not, then $|W_4|\ge d_G(x)-|U|-|W_3|\ge5$, contradicting $|W_4|\le 4$.
So $d_G(x)=\Delta\geq10$. By condition (3),
$d_G(x)\leq|U|+|W_3|+|W_4|\le 4+2+4=10$. Therefore,
$d_G(x)=\Delta=10,|W_4|=4,|U|=4$, and $N_G(x)=U\cup W$.

Without loss of generality, assume that $W_3 = \{w_1, w_2\}$ and $\{u_1,u_2\}\subseteq  N(w_1)$.
If some distinct $u_i, u_j\in U$ have three common neighbors in $W_4$, then these neighbors together with $x, u_i, u_j$ form a copy of $K_{3,3}$.  Assume then that every pair of vertices in $U$ has at most two common neighbors in $W_4$. Since $|E(U, W_4)| =3|W_4|= 12$ and $|U| = 4$, it follows that $|E(\{u_i\}, W_4)|=3$ for each $u_i\in U$. Consequently, there are exactly two vertices in $W_4$, say $w_3, w_4$, adjacent to both $u_1$ and $u_2$. Hence, the subgraph of $B$ induced by $E(\{w_1,w_3,w_4\},\{x,u_1,u_2\})$ is a copy of $K_{3,3}$, completing the proof.
\hfill$\Box$
\medskip

\noindent {\bf The final step.} By Claim \ref{K}, let $Q=K_{3,3}$ be a subgraph in $B$. Recall that every vertex in $V(Q)\cap W$ has degree $3$ or $4$. 
For every edge
$e=uv\in E(Q)$, one end of $e$ lies in $W$ and has degree at most $4$ in $G$,
while the other end lies in $\{x\}\cup U$ and has degree at most $\Delta$ in
$G$. Since each vertex has degree $3$ inside the copy $Q$,
Eq.~(\ref{eq: fQ}) gives
\[
f_Q(e)
=
\Delta+1-\big(d_{\bar{Q}}(u)+d_{\bar{Q}}(v)\big)
\ge
\Delta+1-(\Delta-3)-(4-3)=3.
\]
Thus every half-edge of $Q$ has a remaining list of size at least $3$.
By Lemma~\ref{le: K33},  every signed copy of $K_{3,3}$ is
$3$-list-edge-colorable. Therefore $(Q,\sigma)$ is
$f_Q$-list-edge-colorable, contradicting Claim~\ref{cl: delete}.

This completes the proof of Theorem \ref{maintheorem}.
\hfill$\Box$

\end{document}